\documentclass{amsart}

\usepackage{epsf,latexsym,amsmath,amsfonts,amssymb,subfigure,mathrsfs,graphicx,stmaryrd,eufrak,esint,amscd}
\usepackage{color}

\theoremstyle{plain}
\newtheorem{theorem}{Theorem}[section]
\newtheorem{lemma}[theorem]{Lemma}

\newtheorem{coro}[theorem]{Corollary}

\theoremstyle{definition}

\theoremstyle{remark}
\newtheorem{remark}[theorem]{Remark}

\numberwithin{equation}{section}
\newcommand{\abs}[1]{\lvert#1\rvert}
\newcommand{\labs}[1]{\left\lvert\,#1\,\right\rvert}
\newcommand{\dual}[2]{\left\langle\,#1,#2\right\rangle}
\newcommand{\Lr}[1]{\left(#1\right)}
\newcommand{\lr}[1]{\bigl(#1\bigr)}
\newcommand{\set}[2]{\{\,#1\,\mid\,#2\,\}}
\newcommand{\diff}[2]{\frac{\partial #1}{\partial #2}}
\newcommand{\nm}[2]{\|\,#1\,\|_{#2}}
\newcommand{\jump}[1]{[\![#1]\!]}
\newcommand{\wnm}[1]{|\!|\!|#1|\!|\!|_{\iota,h}}

\newcommand{\mc}[1]{\mathcal{#1}}
\newcommand{\mb}[1]{\mathbb{#1}}

\newcommand{\wh}[1]{\widehat{#1}}
\newcommand{\wt}[1]{\widetilde{#1}}

\def\al{\alpha}

\def\del{\delta}

\def\na{\nabla}
\def\eps{\epsilon}
\def\pa{\partial}
\def\Om{\Omega}
\def\om{\omega}
\def\lam{\lambda}
\def\x{\times}
\def\vpi{\varPi}

\def\md{\mathrm{d}}

\def\C{\mathbb C}
\def\D{\mathbb D}
\def\R{\mathbb R}

\def\dx{\mathrm{d}x}
\def\dy{\mathrm{d}y}

\def\dt{\mathrm{d}\, t}

\DeclareMathOperator{\divop}{\na\cdot}

\DeclareMathOperator{\trop}{\text{tr}}
\def\negint{{\int\negthickspace\negthickspace\negthickspace
\negthinspace -}}
\newcommand{\nn}{\nonumber}
\begin{document}
\title[Nonconforming elements for Strain Gradient Model]{New Nonconforming Elements for Linear Strain Gradient Elastic Model}

\author{Hongliang Li}
\address{Institute of Electronic Engineering, China Academy of Engineering Physics, Mianyang, 621900, China\\
  email: lihongliang@mtrc.ac.cn.}

\author{Pingbing Ming \& Zhong-ci Shi}
\address{LSEC, Institute of Computational Mathematics, AMSS\\
 Chinese Academy of Sciences, No. 55, East Road Zhong-Guan-Cun, Beijing 100190, China\\
 and School of Mathematical Sciences, University of Chinese Academy of Sciences, Beijing 100049, China\\
  email: mpb@lsec.cc.ac.cn, shi@lsec.cc.ac.cn}


\begin{abstract}
Based on a new H$^2-$Korn's inequality, we propose new nonconforming elements for the
linear strain gradient elastic model. The first group of elements are H$^1-$conforming but H$^2-$nonconforming. The tensor product NTW element~\cite{Tai:2001} and the tensor product Specht triangle are two typical representatives. The second element is based on Morley's triangle with a modified elastic strain energy. We proved new interpolation error estimates for all these elements, which are key to prove uniform rates of convergence for the proposed elements. Numerical results are reported and they are consistent with the theoretical prediction.
\end{abstract}

\date{\today}
\maketitle
\section{Introduction}
As an extension of the classical elasticity theory, the strain gradient elasticity theory introduces high order strain tensor and microscopic parameters into the strain energy to characterize the strong size effect of the heterogeneous materials. We refer to~\cite{Eringen:2002} and~\cite{FleckHutchinson:1997} for details of this theory and various strain gradient models. With the increasing interests in the computer simulation of the strain gradient models~\cite{ShuKingFleck:1999, zerovs:20091, Zerovs:20092, Fisher:2010}, it seems that some strain gradient models are less attractive in practice because they have too many material parameters; See, e.g., Mindlin's famous model on the elasticity with microstructure~\cite{Mindlin:1964, MindlinEshel:1968}. By contrast to these models, Aifantis et al~\cite{Altan:1992, Ru:1993} proposed a linear strain gradient elastic model that has only one material parameter. This simplified strain gradient model successfully eliminated the strain singularity of the brittle crack tip field~\cite{Exadaktylos:1996} and we refer to~\cite{Aifantis:2011} for the recent progress of this model.

The strain gradient elastic model of Aifantis is essentially a singularly perturbed elliptic system of fourth order due to the appearance of the strain gradient. C$^1$ finite elements such as Argyris triangle~\cite{ArgyrisFriedScharpf:1968} seems a natural choice for discretizing this model. A drawback of the conforming finite element is that the number of the degrees of freedom is extremely large and high order polynomial has to be used in the basis functions, which is more pronounced for three dimensional problems; See, e.g., the finite element for three-dimensional strain gradient model proposed in~\cite{Zerovs:20092} has $192$ degrees of freedom for the local finite element space.

A common approach to avoid such difficulties is to use the nonconforming finite element. In~\cite{LiMingShi:2017}, we proposed two robust nonconforming H$^2-$finite elements that are H$^1$-conforming. The robustness is understood in the sense that both elements converge uniformly in the energy norm with respect to the small material parameter. The construction of both elements is based on a discrete H$^2$-Korn's inequality, which may be viewed as a higher-order analog of {\sc Brenner's} seminal H$^1$ broken Korn's inequality~\cite{Brenner:2004}. However, the elements proposed in~\cite{LiMingShi:2017} locally belong to a $21$ dimensional subspace of quintic polynomials. In the present work, we aim to develop simpler robust strain gradient elements. We firstly prove a new H$^2-$Korn's inequality and its discrete analog. The novelty of the discrete H$^2-$Korn's inequality is that the jump terms for the gradient field of the piecewise vector filed are dropped from the inequalities. Therefore, the degrees of freedom associated with the gradient fields along each edge are avoided, which greatly simplify the construction of the elements. Based on this observation, we propose two groups of elements. The first group of elements are H$^2-$nonconforming while H$^1$-conforming. The tensor product of the element proposed by {\sc Nillsen, Tai and Winther} (NTW for brevity)~\cite{Tai:2001} and the tensor product of the Specht triangle~\cite{Specht:1988} are two typical examples. Other examples are e.g., the tensor product of the elements in~\cite{GuzmanLeykekhmanNeilan:2012}. Both the NTW element and the Specht triangle have nine degrees of freedom per element and they locally employ quartic polynomials. Another element is a modified Morley's triangle, though Morley's triangle is proved to be diverge for a scalar version of the strain gradient model in~\cite{Tai:2001}, while it converges uniformly if the elastic strain energy is properly modified~\cite{Wang:2006}. Based on the discrete H$^2-$Korn's inequality, we proved that this modified Morley's triangle also converges uniformly for the strain gradient model with a modified elastic strain energy.

The discrete H$^2-$Korn's inequality may also be exploited to develop C$^0$ penalty method to solve the strain gradient elasticity model, such methods for the scalar version of Aifantis' strain gradient model may be found in~\cite{Engel:2002, Brenner:2011}.

It is worth mentioning that alternative approach such as mixed finite element has been employed to discretize this model~\cite{Amanatidou:2002}, while it needs to solve a saddle-point problems and the stable finite element pairs are also very complicate. As to various numerical methods based on reformulations of the strain gradient elastic model, we refer to~\cite{Askes:2002, Askes:2008} and the references therein.


Throughout this paper, the constant $C$ may differ from line to line, while it is independent of the mesh size $h$ and the material parameter $\iota$.
\section{H$^2-$Korn's Inequality}
In this part we prove an H$^2-$Korn's inequality and its discrete analog. We firstly introduce some notations. We shall use the standard notations for Sobolev spaces, norms and seminorms, cf.,~\cite{AdamsFournier:2003}. Let $\Om$ be a bounded polygonal domain, the function space $L^2(\Om)$ consists functions that are square integrable over $\Om$, which is equipped with norm $\nm{\cdot}{L^2(\Om)}$ and the inner product $(\cdot,\cdot)$. Let $H^m(\Om)$ be the Sobolev space of square integrable functions whose weak derivatives up to order $m$ are also square integrable, the corresponding norm is defined by
\(
\nm{v}{H^m(\Om)}^2{:}=\sum_{k=0}^m\abs{v}_{H^k(\Om)}^2
\)
with the seminorm defined by $\abs{v}_{H^k(\Om)}^2{:}=\sum_{\abs{\alpha}=k}\nm{\pa^{\al}v}{L^2(\Om)}$. For a positive number $s$ that is not an integer, $H^s(\Om)$ is the fractional order Sobolev space. Let $m=\lfloor\,s\rfloor$ be the largest integer less than $s$ and $\varrho=s-m$. Then the seminorm $\abs{v}_{H^s(\Om)}$ and the norm $\nm{v}{H^s(\Om)}$ are given by
\begin{align*}
\abs{v}_{H^s(\Om)}^2&=\sum_{\abs{\al}=m}\int_{\Om}\int_{\Om}
\dfrac{\abs{(\pa^{\al})v(x)-(\pa^{\al})v(y)}^2}{\abs{x-y}^{2+2\varrho}}
\dx\,\dy,\\
\nm{v}{H^s(\Om)}^2&=\nm{v}{H^m(\Om)}^2+\abs{v}_{H^s(\Om)}^2.
\end{align*}

By~\cite[\S 7]{AdamsFournier:2003}, the above definition for the fractional order Sobolev space $H^s(\Om)$ is equivalent to the one obtained by interpolation, i.e.,
\[
H^s(\Om)=\left[H^{m+1}(\Om),H^m(\Om)\right]_{\theta}\qquad\text{with\quad}\theta=m+1-s.
\]
In particular, there exists $C$ depends on $\Om$ and $s$ such that
\begin{equation}\label{eq:interineq}
\nm{v}{H^s(\Om)}\le C\nm{v}{H^{m+1}(\Om)}^{1-\theta}\nm{v}{H^m(\Om)}^{\theta}.
\end{equation}

For $s\ge 0, H_0^s(\Om)$ is the closure in $H^s(\Om)$ of the space of $C^\infty(\Om)$ functions with compact supports in $\Om$.

For any vector-valued function $v$, its gradient is a matrix-valued
function with components $(\na v)_{ij}=\pa v_i/\pa x_j$. The symmetric part of a gradient field is also a matrix-valued function defined by $\eps(v)=(\na v+[\na v]^T)/2$. 
The divergence operator applying to a vector field is defined as the trace of $\na v$, i.e., $\divop v=\trop{\na v}=\pa v_i/\pa x_i$.

The Sobolev spaces $[H^m(\Om)]^2, [H_0^m(\Om)]^2$ and  $[L^2(\Om)]^2$ of a vector field can be defined
in a similar manner as their scalar counterparts, this rule equally applies to their inner products and their norms. For the $m-$th order tensors $A,B$, we define the inner product as $A:B=\sum_{i_1,\cdots,i_m}A_{i_1}B_{i_1}
\cdots A_{i_m}B_{i_m}$. Without abuse of notation, we employ $\abs{\cdot}$ to denote the abstract value of a scalar, the
$\ell_2$ norm of a vector, and the Euclidean norm of a matrix.

Throughput this paper, we may drop the subscript $\Om$ whenever there is no confusion occurs.
\subsection{Strain gradient elastic model and H$^2-$Korn inequality}
The strain gradient elastic model in~\cite{Altan:1992,Ru:1993,Exadaktylos:1996} is described by the following boundary
value problem: For $u$ the displacement vector that solves
\begin{equation}\label{eq:sgbvp}
\left\{
\begin{aligned}
(\iota^2\triangle-I)\Lr{\mu\triangle u+(\lambda+\mu)\na\divop u}&=f,\quad&&\text{in\;}\Om,\\
u=0,\pa_n u&=0,\quad&&\text{on\;}\pa\Om.
\end{aligned}\right.
\end{equation}
Here $\lam$ and $\mu$ are the Lam\'e constants, and $\iota$ is the microscopic
parameter such that $0<\iota\le 1$. In particular, we are interested in the regime when $\iota$ is close
to zero. The above boundary value problem may be rewritten into the following variational problem: Find $u\in[H^2_0(\Om)]^2$ such that
\begin{equation}\label{variation:eq}
a(u,v)=(f,v)\quad\text{for all\quad} v\in [H_0^2(\Om)]^2,
\end{equation}
where
\[
a(u,v)=(\C\eps(u),\eps(v))+\iota^2(\D\na\eps(u),\na\eps(v)),
\]
and the fourth-order tensors $\C$ and the sixth-order tensor $\D$ are defined as
\[
\C_{ijkl}=\lam\del_{ij}\del_{kl}+2\mu\del_{ik}\del_{jl}\quad\text{and}\quad
\D_{ijklmn}=\lam\del_{il}\del_{jk}\del_{mn}+2\mu\del_{il}\del_{jm}\del_{ln},
\]
respectively. Here $\delta_{ij}$ is the Kronecker delta function. The third-order tensor $\na\eps(v)$ is defined as $(\na\eps(v))_{ijk}=\eps_{ij,k}$. We only consider the clamped boundary condition in this paper, the discussion on
other boundary conditions can be found in~\cite{Altan:1992,Ru:1993,Exadaktylos:1996}.

The variational problem~\eqref{variation:eq} is well-posed if and only if the bilinear form $a$ is coercive over $[H_0^2(\Om)]^2$, which depends an H$^2-$Korn's inequality: For any $v\in[H^2(\Om)\cap H_0^1(\Om)]^2$, there exists $C$ such that
\begin{equation}\label{eq:h2korn}
\nm{\na\eps(v)}{L^2}^2+\nm{\eps(v)}{L^2}^2\ge C\nm{\na v}{H^1}^2.
\end{equation}
This inequality was proved in~\cite{Necas:1967} in a very general form
with the aid of the so-called {\em Necas' Lemma}, however, the explicit constant $C$ is unknown. Another proof given in~\cite{AbelMoraMuller:2011} and~\cite{LiMingShi:2017} for~\eqref{eq:h2korn} is based on the fact that $\pa_i v\in [H^1(\Om)]^{2\x 2}$ and
the community property of strain operator $\eps$ and the partial derivative operator $\pa$. The explicit constant $C$ is clarified in~\cite[Theorem 1]{LiMingShi:2017} when $v\in [H_0^2(\Om)]^2$, while this spacial case suffices for the coercivity of the bilinear form.~\footnote{The authors in~\cite{AbelMoraMuller:2011} and~\cite{LiMingShi:2017} only proved~\eqref{eq:h2korn} under certain special boundary conditions, while their proofs are easily adapted to this case.}

In this part we give a new proof for~\eqref{eq:h2korn}. The precise form is~\eqref{eq:korn}. Our proof relies on the fact that the strain gradient fully controls the Hessian of the displacement; See cf.~\eqref{eq:algkorn}. This fact will be further exploited to prove a discrete analog of~\eqref{eq:h2korn}, which is key to design robust finite elements for the strain gradient elastic model as shown in~\cite{LiMingShi:2017}.
\begin{lemma}\label{lema:h2korn}
For any $v\in [H_0^1(\Om)]^2,\eps(v)\in [L^2(\Om)]^{2\x 2}$ and $\na\eps(v)\in [L^2(\Om)]^{2\x 2\x 2}$, there holds
$\na v\in [H^1(\Om)]^{2\x 2}$ and
\begin{equation}\label{eq:korn}
\nm{\eps(v)}{L^2}^2+\nm{\na\eps(v)}{L^2}^2\ge (1-1/\sqrt2)\nm{\na v}{H^1}^2.
\end{equation}
\end{lemma}

\begin{remark}
The above H$^2-$Korn's inequality~\eqref{eq:korn}, and the proof below are also valid for $d=3$, the details will be presented in another work.
\end{remark}

\begin{proof}
The core of the proof is the following algebraic inequality:
\begin{equation}\label{eq:algkorn}
\abs{\na\eps(v)}^2\ge (1-1/\sqrt2)\abs{\na^2 v}^2.
\end{equation}

Integrating~\eqref{eq:algkorn} over domain $\Om$, we obtain
\begin{equation}\label{eq:2ndder}
\nm{\na\eps(v)}{L^2}^2\ge(1-1/\sqrt{2})\nm{\na^2 v}{L^2}^2,
\end{equation}
which together with the first Korn's inequality~\cite{Korn:1908, Korn:1909}:
\begin{equation}\label{eq:1stkorn}
2\nm{\eps(v)}{L^2}^2\ge\nm{\na v}{L^2}^2\qquad\text{for all\quad}v\in [H_0^1(\Om)]^2
\end{equation}
implies~\eqref{eq:korn}.

To prove~\eqref{eq:algkorn}, we note
\begin{align}\label{eq:defsg}
\abs{\na\eps(v)}^2&=\labs{\dfrac{\pa^2 v_1}{\pa x_1^2}}^2+\labs{\dfrac{\pa^2 v_1}{\pa x_1\pa x_2}}^2
+\labs{\dfrac{\pa^2 v_2}{\pa x_1\pa x_2}}^2+\labs{\dfrac{\pa^2 v_2}{\pa x_2^2}}^2\nn\\
&\quad+\dfrac12\labs{\dfrac{\pa^2 v_1}{\pa x_1\pa x_2}+\dfrac{\pa^2 v_2}{\pa x_1^2}}^2
+\dfrac12\labs{\dfrac{\pa^2 v_2}{\pa x_1\pa x_2}+\dfrac{\pa^2 v_1}{\pa x_2^2}}^2.
\end{align}
By
\[
a^2+\dfrac12(a+b)^2\ge\Lr{1-\dfrac{1}{\sqrt2}}(a^2+b^2), \qquad a,b\in\R.
\]
We conclude that
\begin{align*}
\labs{\dfrac{\pa^2 v_1}{\pa x_1\pa x_2}}^2+\dfrac12\labs{\dfrac{\pa^2 v_1}{\pa x_1\pa x_2}+\dfrac{\pa^2 v_2}{\pa x_1^2}}^2
\ge\Lr{1-\dfrac1{\sqrt2}}\Lr{\labs{\dfrac{\pa^2 v_1}{\pa x_1\pa x_2}}^2+\labs{\dfrac{\pa^2 v_2}{\pa x_1^2}}^2},\\
\labs{\dfrac{\pa^2 v_2}{\pa x_1\pa x_2}}^2+\dfrac12\labs{\dfrac{\pa^2 v_2}{\pa x_1\pa x_2}+\dfrac{\pa^2 v_1}{\pa x_2^2}}^2
\ge\Lr{1-\dfrac1{\sqrt2}}\Lr{\labs{\dfrac{\pa^2 v_2}{\pa x_1\pa x_2}}^2+\labs{\dfrac{\pa^2 v_1}{\pa x_2^2}}^2}.
\end{align*}
Substituting the above two inequalities into~\eqref{eq:defsg}, we obtain~\eqref{eq:algkorn}. This completes the proof.
\end{proof}
%

A direct consequence of Lemma~\ref{lema:h2korn} is the following full H$^2-$Korn's inequality.
\begin{coro}
Let $\Om\subset\R^2$ be a domain such that the following Korn's inequality is valid for
any vector field $v\in [L^2(\Om)]^2$ and $\eps(v)\in[L^2(\Om)]^{2\x 2}$,
\begin{equation}\label{eq:fullkorn}
\nm{v}{L^2}+\nm{\eps(v)}{L^2}\ge C(\Om)\nm{v}{H^1}.
\end{equation}
If $v\in [L^2(\Om)]^2,\eps(v)\in[L^2(\Om)]^{2\x 2}$ and $\na\eps(v)\in [L^2(\Om)]^{2\x 2\x 2}$,
then $v\in [H^2(\Om)]^2$ and
\begin{equation}\label{eq:fullkorn}
\nm{v}{L^2}+\nm{\eps(v)}{L^2}+\nm{\na\eps(v)}{L^2}\ge \Lr{C(\Om)\wedge(1-1/\sqrt2)}\nm{v}{H^2}.
\end{equation}
\end{coro}

The following regularity results for the solution of Problem~\eqref{variation:eq} is crucial.
\begin{lemma}\label{lemma:reg}
There exists $C$ that may depend on $\Om$ but independent of $\iota$ such that
\begin{equation}\label{eq:regularity}
\abs{u}_{H^2}+\iota\abs{u}_{H^3}\le C\iota^{-1/2}\nm{f}{L^2},
\end{equation}
and
\begin{equation}\label{eq:limiterr}
\nm{u-u_0}{H^1}\le C\iota^{1/2}\nm{f}{L^2},
\end{equation}
where $u_0\in [H_0^1(\Om)]^2$ satisfies
\begin{equation}\label{eq:limit}
(\C\eps(u_0),\eps(v))=(f,v)\qquad \text{for all\quad}v\in[H_0^1(\Om)]^2.
\end{equation}

Moreover, we have the estimates
\begin{equation}\label{eq:estfrac0}
\abs{u}_{H^2}\abs{u}_{H^3}\le C\iota^{-2}\nm{f}{L^2}^2,
\end{equation}
and
\begin{equation}\label{eq:regfrac}
\nm{u}{H^{3/2}}\le C\nm{f}{L^2}.
\end{equation}
\end{lemma}

\begin{proof}
Proceeding along the same line in~\cite[Appendix]{SchatzWahlbin:1983}; See also~\cite[Lemma 5.1]{Tai:2001}, we may prove the estimates~\eqref{eq:regularity} and~\eqref{eq:limiterr}.

By~\eqref{eq:regularity}, we may have
\[
\abs{u}_{H^2}\le C\iota^{-1/2}\nm{f}{L^2}\qquad\text{and\qquad}
\abs{u}_{H^3}\le C\iota^{-3/2}\nm{f}{L^2},
\]
which immediately implies~\eqref{eq:estfrac0}.

Using~\eqref{eq:regularity} and the regularity estimate for Problem~\eqref{eq:limit}
\[
\nm{u_0}{H^2}\le C\nm{f}{L^2},
\]
we obtain
\[
\abs{u-u_0}_{H^2}\le\abs{u}_{H^2}+\abs{u_0}_{H^2}\le C\iota^{-1/2}\nm{f}{L^2}.
\]
Using~\eqref{eq:limiterr}, and noting that $\iota<1$, we obtain
\begin{equation}\label{eq:limiterr2}
\nm{u-u_0}{H^2}\le C\Lr{\iota^{1/2}+\iota^{-1/2}}\nm{f}{L^2}\le C\iota^{-1/2}\nm{f}{L^2}.
\end{equation}
Interpolating the above inequality and~\eqref{eq:limiterr}, we obtain
\[
\nm{u-u_0}{H^{3/2}}\le C\nm{f}{L^2}.
\]
Using the interpolation inequality~\eqref{eq:interineq}, we obtain
\[
\nm{u_0}{H^{3/2}}\le C\nm{u_0}{H^1}^{1/2}\nm{u_0}{H^2}^{1/2}\le C\nm{f}{L^2}
\]
A combination of the above two inequalities yields~\eqref{eq:regfrac}.
\end{proof}
\section{Discrete H$^2-$Korn's Inequality}
Let $\mc{T}_h$ be a triangulation of $\Omega$ with maximum mesh size $h$. We assume all elements in $\mc{T}_h$ are shape-regular in the sense of Ciarlet and Raviart~\cite{Ciarlet:1978}, i.e., there exists a constant $\gamma$ such that
$h_K/\rho_K\le\gamma$, where $h_K$ is the diameter of element $K$, and $\rho_K$ is the diameter of the smallest ball inscribed into $K$, and $\gamma$ is the so-called chunkness parameter~\cite{BrennerScott:2008}. Denote the set of all the edges in $\mc{T}_h$ as $\mc{E}_h$. For any $m\in\mb{N}$, the space of piecewise $[H^m(\Om,\mc{T}_h)]^2$ vector fields is defined by
\[
[H^m(\Om,\mc{T}_h)]^2{:}=\set{v\in [L^2(\Om)]^2}{v|_K\in[H^m(K)]^2,\quad\forall K\in\mc{T}_h},
\]
which is equipped with the broken norm
\[
\nm{v}{H_h^k}{:}=\nm{v}{L^2}+\sum_{k=1}^m\nm{\na^k_h v}{L^2},
\]
where
\(
\nm{\na^k_h v}{L^2}^2=\sum_{K\in\mc{T}_h}\nm{\na^k v}{L^2(K)}^2
\)
with $(\na^k_h v)|_K=\na^k(v|_K)$. Moreover, $\eps_h(v)=(\na_h v+[\na_h v]^T)/2$. For any $v\in H^m(\Om,\mc{T}_h)$, we denote by $\jump{v}$ the jump of $v$ across the edge of $\mc{E}_h$.

Let $\mb{P}_k$ be the Lagrange element of order $k$. Denote $b_K$ as the bubble function of $K$, i.e., $b_K=\lam_1\lam_2\lam_3$ with $\lam_i$ the barycentric coordinate associated with the vertices $a_i$ for $i=1,2,3$.
The main result of this part is the following discrete H$^2-$Korn's inequality.
\begin{theorem}\label{thm:hkorn}
For any $v\in[H^2(\Omega,\mc{T}_h)]^2$, there exits $C$ that depends on $\Om$ and $\gamma$ but independent of $h$ such that
\begin{equation}\label{eq:diskorn}
\begin{aligned}
\nm{v}{H_h^2}^2\le C\biggl(\nm{\na_h\eps(v)}{L^2}^2&+\nm{\eps_h(v)}{L^2}^2+\nm{v}{L^2}^2\\
&\quad+\sum_{e\in\mc{E}_h}h_{e}^{-1}\nm{\jump{\Pi_{e}v}}{L^2(e)}^2\biggr),
\end{aligned}
\end{equation}
where $\Pi_{e}:[L^2(e)]^2\mapsto [P_{1,-}(e)]^2$ is the $L^2$ projection and
\[
[P_{1,-}(e)]^2{:}=\set{v\in [P_1(e)]^2}{v\cdot t\in\text{RM}(e)},
\]
where $t$ is the tangential vector of edge $e$, and $\text{RM}(e)$ is the infinitesimal rigid motion on $e$.
\end{theorem}

The inequality~\eqref{eq:diskorn} improves the original discrete H$^2-$Korn inequality proved in~\cite{LiMingShi:2017} by removing the jump term
\[
\sum_{i=1}^2\sum_{e\in\mc{E}_h}h_{e}^{-1}\nm{\jump{\Pi_{e}(v_{,i})}}{L^2(e)}^2
\]
from the right-hand side of~\eqref{eq:diskorn}. This jump term stands for the jump of the gradient tensor of the field $v$ across the element boundary. This would greatly simplify the construction of the robust strain gradient elements as shown in the next two parts.
\vskip .5cm
\noindent{\em Proof of Theorem~\ref{thm:hkorn}\;}
Integrating~\eqref{eq:algkorn} over element $K\in\mc{T}_h$, we obtain,
\[
\nm{\na\eps(v)}{L^2(K)}^2\ge(1-1/\sqrt2)\nm{\na^2 v}{L^2(K)}^2,
\]
which together with the following discrete {\em Korn's inequality} proved by {\sc Mardal and Winther}~\cite{MardalWinther:2006}
\begin{equation}\label{eq:discretekorn}
\nm{v}{H_h^1}^2\le C\Lr{\|\eps_h(v)\|_{L^2}^2+\nm{v}{L^2}^2
+\sum_{e\in\mc{E}_h}h_{e}^{-1}\nm{\jump{\Pi_{e}v}}{L^2(e)}^2}
\end{equation}
implies~\eqref{eq:diskorn}.
\qed

We shall frequently use the following versions of the trace inequalities.
\begin{lemma}
For any Lipschitz domain $D$, there exists $C$ that depends on $D$ such that
\begin{equation}\label{eq:tracedomain}
\nm{v}{L^2(\pa D)}\le C\nm{v}{L^2(D)}^{1/2}\nm{v}{H^1(D)}^{1/2}.
\end{equation}

For any element $K\in\mc{T}_h$, there exists $C$ independent of $h_K$, but depends on $\gamma$ such that
\begin{equation}\label{eq:eletrace}
\nm{v}{L^2(\pa K)}\le C\Lr{h_K^{-1/2}\nm{v}{L^2(K)}+\nm{v}{L^2(K)}^{1/2}\nm{\na v}{L^2(K)}^{1/2}}.
\end{equation}

If $v\in\mb{P}_m(K)$, then there exists a constant $C$ independent of $v$, but may depend on $\gamma$ and $m$ such that
\begin{equation}\label{eq:polytrace}
\nm{v}{L^2(\pa K)}\le Ch_K^{-1/2}\nm{v}{L^2(K)}.
\end{equation}
\end{lemma}

The multiplicative type trace inequality~\eqref{eq:tracedomain} may be found in~\cite{Grisvard:1985}, while~\eqref{eq:eletrace} is a direct consequence of~\eqref{eq:tracedomain}. The third trace inequality is a combination of~\eqref{eq:eletrace} and the inverse inequality for any polynomial $v\in\mb{P}_m(K)$.

We shall use the following lemma.
\begin{lemma}\label{lema:skeleton}
For any $\eps>0$, let $\Om_{\eps}{:}=\set{x\in\Om}{\text{dist}(x,\pa\Om)\le\eps}$. Then for any $v\in H^1(\Om)$, there exists $C$ independent of $\eps$ but depends on $\Om$ such that
\begin{equation}\label{eq:skeleton}
\nm{v}{L^2(\Om_\eps)}\le C\sqrt\eps\nm{v}{H^1(\Om)}.
\end{equation}
\end{lemma}
\section{The First Nonconforming Elements}
Motivated by the discrete Korn's inequality~\eqref{eq:diskorn}, we shall construct new finite elements to approximate the variational problem~\eqref{eq:sgbvp}. The finite element space is defined by $V_h=[X_h]^2$ with
\[
X_h{:}=\set{v\in [H_0^1]^2}{v|_K\in P_K\;\text{for all}\;K\in\mc{T}_h},
\]
where $P_K$ will be specified later on. The corresponding finite element spaces with homogeneous boundary condition are defined by $V_h^0=[X_h^0]^2$ with
\[
X_h^0{:}=\set{v\in X_h}{\text{all degrees of freedom associated with boundary are zeros}}.
\]

Given $V_h^0$, we find $u_h\in V_h^0$ such that
\begin{equation}\label{eq:disvara}
a_h(u_h,v)=(f,v)\quad\text{for all\quad} v\in V_h^0,
\end{equation}
where the bilinear form $a_h$ is defined for any $v,w\in V_h^0$ as
\[
a_h(v,w){:}=(\C\eps(v),\eps(w))+\iota^2(\D\na_h\eps(v),\na_h\eps(w))
\]
with
\[
(\D\na_h\eps(v),\na_h\eps(w)){:}=\sum_{K\in\mc{T}_h}\int_K\D\na\eps(v)\na\eps(w)\dx.
\]
The energy norm is defined as $\wnm{v}=\nm{\na_h v}{L^2}+\iota\nm{\na_h^2 v}{L^2}$.
%
\subsection{The Tensor product of the element of {\sc Nillsen, Tai and Winther}}
The first element is a tensor product of the element proposed by {\sc Nillsen, Tai and Winther}, which has been exploited to approximate a fourth order elliptic singular perturbation problem~\cite{Tai:2001}. NTW element is defined by a finite element triple $(K,P_K,\Sigma_K)$~\cite{Ciarlet:1978} as
\[
\left\{\begin{aligned}
P_K&=\mb{P}_2(K)+b_K\mb{P}_1(K),\\
\Sigma_K&=\{p(a_i),p(b_i),\negint_{e_i}\pa_n p,\;i=1,2,3\},
\end{aligned}\right.
\]
where for $i=1,2,3$, $a_i$ are three vertices of $K$, and $e_i$ are the edges opposite to the vertices $a_i$, and $b_i$ are the midpoints of edge $e_i$.

For each element $K$, the corresponding NTW interpolation operator $\vpi_K$ is defined for $i=1,2,3$,
\[
\vpi_Kv(a_i)=v(a_i),\quad \vpi_Kv(b_i)=v(b_i),\quad\int_{e_i}\diff{\vpi_K v}{n}=\int_{e_i}\diff{v}{n}.
\]

It is easy to verify that
\begin{equation}\label{eq:ntwinter}
\vpi_Kv=\sum_{i=1}^3\Lr{v(a_i)\phi_i+v(b_i)\varphi_i+\Lr{\negint_{e_i}\pa_n v}\psi_i},
\end{equation}
where
\begin{align*}
\phi_i&=\lam_i(2\lam_i-1)-6b_K(2\lam_i-1)+6b_K\sum_{j\not= i}\dfrac{\na\lam_i\cdot\na\lam_j}{\abs{\na\lam_j}^2}(2\lam_j-1),\\
\varphi_i&=4\lam_j\lam_k+12b_K(1-4\lam_i),\quad\psi_i=6b_K(2\lam_i-1)/\abs{\na\lam_i}.
\end{align*}

The interpolate error estimate for $\vpi_K$ is given in the following lemma.
%
\begin{lemma}\label{lema:ntw}
There exists $C$ that depends on $\gamma$ such that
\begin{equation}\label{eq:ntwinter1}
\sum_{j=0}^2h_K^j\nm{\na^j(v-\vpi_K v)}{L^2(K)}\le Ch_K^m\nm{\na^mv}{L^2(K)},
\quad m=2,3.
\end{equation}
and
\begin{equation}\label{eq:ntwinter2}
\nm{\na(v-\vpi_K v)}{L^2(K)}\le Ch_K^{1/2}\nm{\na v}{L^2(K)}^{1/2}\nm{\na^2 v}{L^2(K)}^{1/2}.
\end{equation}
\end{lemma}

The first estimate~\eqref{eq:ntwinter1} is quite standard. The second estimate~\eqref{eq:ntwinter2} is crucial for deriving the uniform error bound, which was proved in~\cite{Tai:2001} by a standard scaling argument. Unfortunately, the presence of the degrees of freedom $\int_e \pa_n v$ prevents the affine-equivalence~\cite{CiarletRaviart:1972} of the corresponding interpolant. Therefore, the standard scaling argument fails for deriving such interpolation estimate. A remedy is to introduce the affine relative of the NTW element~\eqref{eq:affinentw} as in~\cite{LascauxLesaint:1975} for deriving the interpolation error estimate of Morley's triangle; See also~\cite{CiarletRaviart:1972}.

The finite element triple $(K,P_K,\Sigma_K)$ for the affine relative of the NTW element is almost the same with those
of NTW except that the degrees of freedom is changed into
\[
\wt{\Sigma}_K=\{p(a_i),p(b_i),\negint_{e_i}(m_i\cdot\na)p,i=1,2,3\},
\]
where $m_i$ is the vector $a_ib_i$. The corresponding interpolation operator $\wt{\vpi}_K$ is defined as: for $i=1,2,3$,
\begin{equation}\label{eq:affinentw}
\wt{\vpi}_Kv(a_i)=v(a_i),\quad\wt{\vpi}_Kv(b_i)=v(b_i)\quad
\int_{e_i}(m_i\cdot\na)\wt{\vpi}_K v=\int_{e_i}(m_i\cdot\na)v.
\end{equation}
It is clear to write
\[
\wt{\vpi}_Kv=\sum_{i=1}^3\Lr{v(a_i)\wt{\phi}_i+v(b_i)\wt{\varphi}_i+\negint_{e_i}(m_i\cdot\na v)\wt{\psi}_i},
\]
where
\begin{align*}
\wt{\phi}_i&=\lam_i(2\lam_i-1)+6b_K(1-\lam_i),\quad\wt{\varphi}_i=4\lam_j\lam_k+12b_K(1-4\lam_i),\\
\wt{\psi}_i&=6b_K(2\lam_i-1).
\end{align*}

The relationship between these two interpolation operator reads as
\begin{lemma}\label{lema:affine}
For any $v\in H^2(K)$, there holds
\begin{equation}\label{eq:interiden}
\vpi_K v=\wt{\vpi}_K v.
\end{equation}
\end{lemma}

The proof is quite long but straightforward, we postpone it to the Appendix.

\vskip .5cm\noindent\textbf{Proof of Lemma~\ref{lema:ntw}\;}
By~\eqref{eq:interiden}, we only need to prove~\eqref{eq:ntwinter1} and~\eqref{eq:ntwinter2} for $\wt{\vpi}_K$. The first estimate~\eqref{eq:ntwinter1} is straightforward because $\wt{\vpi}_K$ is affine-invariant.

Using a standard scaling argument, we obtain
\begin{align*}
\nm{\na(v-\wt{\vpi}_K v)}{L^2(K)}&\le C\nm{\wh{\na}(\wh{v}-\wh{\wt{\vpi}}_{\wh{K}}\wh{v})}{L^2(\wh{K})}\\
&\le C\Lr{\nm{\wh{\na}\wh{v}}{L^2(\wh{K})}+\nm{\wh{v}}{L^\infty(\wh{K})}+\nm{\wh{\na}\wh{v}}{L^2(\pa\wh{K})}}.
\end{align*}
The left-hand side is invariant if we replace $v$ by $v-c$ for any constant $c$, then, we may rewrite the above
inequality into
\[
\nm{\wh{\na}(\wh{v}-\wh{\wt{\vpi}}_{\wh{K}}\wh{v})}{L^2(\wh{K})}
\le C\Lr{\nm{\wh{\na}\wh{v}}{L^2(\wh{K})}+\nm{\wh{v}-c}{L^\infty(\wh{K})}+\nm{\wh{\na}\wh{v}}{L^2(\pa\wh{K})}}.
\]
By Sobolev imbeding theorem $W^{1,p}(\wh{K})\hookrightarrow L^\infty(\wh{K})$~\cite{AdamsFournier:2003} for any $p>2$, we obtain that for any $2<p<4$ and, there holds
\[
\nm{\wh{v}-c}{L^\infty(\wh{K})}\le C\Lr{\nm{\wh{v}-c}{L^p(\wh{K})}+\nm{\wh{\na}\wh{v}}{L^p(\wh{K})}}\le C\nm{\wh{\na}\wh{v}}{L^p(\wh{K})},
\]
where we have used the Poincar\'e inequality over $\wh{K}$ in the last step.

Next, by interpolation inequality and the Sobolev imbedding inequality $H^1(\wh{K})\hookrightarrow L^q(\wh{K})$ with $q=2p/(4-p)$, there holds
\[
\nm{\wh{\na}\wh{v}}{L^p(\wh{K})}\le\nm{\wh{\na}\wh{v}}{L^2(\wh{K})}^{1/2}
\nm{\wh{\na}\wh{v}}{L^q(\wh{K})}^{1/2}\le C\nm{\wh{\na}\wh{v}}{L^2(\wh{K})}^{1/2}\nm{\wh{\na}\wh{v}}{H^1(\wh{K})}^{1/2}.
\]

Combining the above three inequalities and using~\eqref{eq:tracedomain} to bound $\nm{\wh{\na}\wh{v}}{L^2(\pa\wh{K})}$, we obtain
\[
\nm{\wh{\na}(\wh{v}-\wh{\wt{\vpi}}_{\wh{K}}\wh{v})}{L^2(\wh{K})}\le C\nm{\wh{\na}\wh{v}}{L^2(\wh{K})}^{1/2}\nm{\wh{\na}\wh{v}}{H^1(\wh{K})}^{1/2}.
\]
Note $\wh{\wt{\vpi}}$ is affine-invariant for $\mb{P}_1$, then we have
\[
\nm{\wh{\na}(\wh{v}-\wh{\wt{\vpi}}_{\wh{K}}\wh{v})}{L^2(\wh{K})}\le C\inf_{p\in\mb{P}_1(\wh{K})}\nm{\wh{\na}(\wh{v}-p)}{L^2(\wh{K})}^{1/2}\nm{\wh{\na}(\wh{v}-p)}{H^1(\wh{K})}^{1/2}.
\]
We take $p$ as the Galerkin projection of $\wh{v}$ in the sense that $p\in\mb{P}_1(\wh{K})$ satisfying
\[
\int_{\wh{K}}\wh{\na}(\wh{v}-p)\wh{\na}\wh{w}=0\quad\text{for all\quad}\wh{w}\in\mb{P}_1(\wh{K}).
\]
By error estimate for Galerkin projection~\cite{Braess:2007}, we have
\[
\nm{\wh{\na}(\wh{v}-p)}{L^2(\wh{K})}\le\nm{\wh{\na}\wh{v}}{L^2(\wh{K})}\quad\text{and\quad}
\nm{\wh{\na}(\wh{v}-p)}{H^1(\wh{K})}\le C\nm{\wh{\na}^2\wh{v}}{L^2(\wh{K})}.
\]
Combining the above two estimates, we obtain
\[
\nm{\wh{\na}(\wh{v}-\wh{\wt{\vpi}}_{\wh{K}}\wh{v})}{L^2(\wh{K})}\le C\nm{\wh{\na}\wh{v}}{L^2(\wh{K})}^{1/2}
\nm{\wh{\na}^2\wh{v}}{L^2(\wh{K})}^{1/2}.
\]
A standard scaling argument yields~\eqref{eq:ntwinter2}.
\qed

A direct consequence of Lemma~\ref{lema:ntw} is the following approximation result in the energy norm.
\begin{coro}\label{coro:ntwnmint}
There holds
\begin{equation}\label{eq:ntwwnmint}
\inf_{v\in V_h^0}\wnm{u-v}\le Ch^{1/2}\nm{f}{L^2}.
\end{equation}
\end{coro}

\begin{proof}
Let $v=\vpi_h u=(\vpi_hu_1,\vpi_hu_2)$ with $(\vpi_h)_K=\vpi_K$. Using~\eqref{eq:ntwinter2}, we obtain
\begin{align*}
\nm{\na(u-\vpi_hu)}{L^2}&\le\nm{\na(I-\vpi_h)(u-u_0)}{L^2}
+\nm{\na(u_0-\vpi_hu_0)}{L^2}\\
&\le Ch^{1/2}\nm{\na(u-u_0)}{L^2}^{1/2}\nm{\na^2(u-u_0)}{L^2}^{1/2}+Ch\nm{\na^2u_0}{L^2}.
\end{align*}
Substituting~\eqref{eq:limiterr} and~\eqref{eq:limiterr2} into the above inequality, we obtain
\begin{equation}\label{eq:ntwsginter}
\nm{\na(u-\vpi_hu)}{L^2}\le Ch^{1/2}\nm{f}{L^2}.
\end{equation}

Using~\eqref{eq:ntwinter1} with $j=m=2$, we have
\[
\nm{\na^2(u-\vpi_h u)}{L^2}\le C\nm{\na^2u}{L^2}.
\]
Using~\eqref{eq:ntwinter1} with $j=2$ and $m=3$, we have
\[
\nm{\na^2(u-\vpi_h u)}{L^2}\le Ch\nm{\na^3u}{L^2}.
\]
Interpolate between the above two inequalities, we obtain
\[
\nm{\na^2(u-\vpi_h u)}{L^2}\le Ch^{1/2}\nm{\na^2u}{L^2}^{1/2}\nm{\na^3u}{L^2}^{1/2}.
\]
Using the regularity estimate~\eqref{eq:estfrac0}, we obtain
\[
\iota\nm{\na^2(u-\vpi_hu)}{L^2}\le Ch^{1/2}\nm{f}{L^2},
\]
which together with~\eqref{eq:ntwsginter} implies~\eqref{eq:ntwwnmint}.
\end{proof}
\subsection{The Tensor product of the Specht triangle}
 The second one is the tensor product of the Specht triangle~\cite{Specht:1988}. This element
is a successful plate bending element,~\textit{which passes all the patch tests and performs excellently, and is one of the
best thin plate triangles with $9$ degrees of freedom that currently available}~\cite[citation in p. 345]{ZienkiewiczTaylor:2009}.

The Specht triangle is defined by the finite element triple $(K,P_K,\Sigma_K)$ as
\[
\left\{\begin{aligned}
P_K&=Z_K+b_K\mb{P}_1(K),\\
\Sigma_K&=\{p(a_i),\pa_x p(a_i),\pa_yp(a_i),i=1,2,3\}
\end{aligned}\right.
\]
with three extra constraints
\begin{equation}\label{eq:spechtconstraint}
\int_{e_i}P_2\pa_n v=0,\quad i=1,2,3,
\end{equation}
where $P_2$ is the Legendre polynomial of second order on $e_i$. Here $Z_K$ is the Zienkiewicz space~\cite{BazeleyCheungIronsZienkiewicz:1965} defined by
\[
Z_K=\mb{P}_2(K)+\text{Span}\{\lam_1^2\lam_2-\lam_2^2\lam_1,\lam_2^2\lam_3-\lam_3^2\lam_2,\lam_3^2\lam_1-\lam_1^2\lam_3\}.
\]

The standard interpolate estimate for Specht triangle reads as; cf.~\cite{CiarletRaviart:1972}
\begin{equation}\label{eq:spechtinter1}
\sum_{j=0}^3h_K^j\nm{\na^j(v-\vpi_K v)}{L^2(K)}\le Ch_K^3\nm{\na^3 v}{L^2(K)},
\end{equation}
where the interpolation operator $\vpi_K: H^3(K)\to P_K$ is defined for $i=1,2,3$ as
\[
\vpi_K v(a_i)=v(a_i),\quad \na\vpi_K v(a_i)=\na v(a_i).
\]
This interpolant is not bounded in $H^2(K)$, which is even not well-defined for functions in $H^2(K)$. In what follows we define a regularized interpolant that is bounded in $H^1(K)$, which is crucial for deriving the uniform error estimate.
\begin{lemma}
There exists an operator $I_h^{\,0}:H_0^1(\Om)\to X_h^0$ such that
\begin{enumerate}
\item For $m=0,1$, there holds
\begin{equation}\label{eq:spechtstable}
\nm{v-I_h^{\,0} v}{H^m(\Om)}\le Ch^{1-m}\nm{\na v}{L^2(\Om)}.
\end{equation}

\item If $v\in H^s(\Om)\cap H_0^2(\Om)$ with $s=2,3$, then for $0\le m\le s$ with $m\in\mb{Z}$,
\begin{equation}\label{eq:spechtregint1}
\nm{v-I_h^{\,0} v}{H^m_h}\le Ch^{s-m}\abs{v}_{H^s(\Om)},
\end{equation}

\item If $v\in H^s(\Om)\cap H_0^1(\Om)$ with $s=2,3$, but $v\not\in H_0^2(\Om)$, then for $0\le m\le s$ with $m\in\mb{Z}$,
\begin{equation}\label{eq:spechtregint2}
\nm{v-I_h^{\,0} v}{H^m_h}\le C\Lr{h^{s-m}\abs{v}_{H^s(\Om)}+h^{3/2-m}\nm{\na v}{H^1(\Om)}}.
\end{equation}
\end{enumerate}
\end{lemma}

The above estimates are in the same spirit of~\cite[Lemma 2]{GuzmanLeykekhmanNeilan:2012}, while the construction is different and the last estimate~\eqref{eq:spechtregint2} is new.

A useful consequence of the above lemma is
\begin{equation}\label{eq:spechtint3}
\nm{v-I_h^0v}{H^1}\le Ch^{s-1}\nm{v}{H^s}\qquad\text{for all\;}v\in H^s(\Om)\cap H_0^2(\Om)\quad 1\le s\le 2.
\end{equation}
This may be proved by interpolating between~\eqref{eq:spechtstable} with $m=2$ and~\eqref{eq:spechtregint1} with $m=1,s=2$.
\begin{remark}
A different regularized interpolant was constructed for the Specht triangle in~{\sc Wang, Shi and Xu}~\cite{WangShiXu:2007}~\footnote{Note that the Zienkiewicz-type element proposed in~\cite{WangShiXu:2007} coincides with the Specht triangle in two dimension.} and the estimate~\eqref{eq:spechtregint1} was also proved for $s=3$. It is unclear whether the interpolant is $H^1$ bounded and the estimate~\eqref{eq:spechtregint2} is valid.
\end{remark}
\begin{proof}
Define
$\vpi_h$ locally as
\begin{equation}\label{eq:construct}
\left\{
\begin{aligned}
\vpi_h v(a)&=v(a)\quad\text{for all vertices\;}a,\\
(\na\vpi_h v)(a)&=\dfrac1{\abs{\om_a}}\sum_{K'\in\om_a}(\na v|_{K'})(a)\quad\text{for all interior vertices},\\
(\na\vpi_h v)(a)&=0,\quad\text{for all vertices on the boundary},
\end{aligned}\right.
\end{equation}
where $\om_a$ is the set of the triangles in $\mc{T}_h$ that share a common vertex $a$ and $\abs{\om_a}$ is the number of the triangles inside $\om_a$.

Define
\[
I_h^0{:}=\vpi_h\circ\vpi_C,
\]
where $\vpi_C:H_0^1(\Om)\to L_h$ is the {\sc Scott-Zhang} interpolant~\cite{ScottZhang:1990} with $L_h\subset H_0^1(\Om)$ a $\mb{P}_2$ Lagrangian finite element space. By the definition of $X_h^0$, the interpolant $I_h^0$ is well-defined.

For any $\phi\in P_K$, a standard scaling argument yields
\begin{equation}\label{eq:eqnm}
\abs{\phi}_{H^m(K)}^2\le Ch_K^{2-2m}\sum_{i=1}^3\Lr{\abs{\phi(a_i)}^2+h_K^2\abs{\na\phi(a_i)}^2}\qquad m=0,1,2.
\end{equation}
Substituting $\phi=w-\vpi_hw$ with $w=\vpi_Cv$ into the above inequality and noting $\phi(a_i)=0$, we obtain
\begin{align*}
\abs{\phi}_{H^m(K)}^2&\le Ch_K^{4-2m}\sum_{i=1}^3\dfrac1{\abs{\om_{a_i}}^2}\labs{\sum_{K'\in\om_{a_i}}\bigl[(\na w|_K)(a_i)-(\na w|_{K'})(a_i)\bigr]}^2\\
&\le Ch_K^{4-2m}\sum_{i=1}^3\sum_{\substack{K',K^{''}\\ K'\text{and\;}K^{''}\text{\;share
an edge}}}\abs{\na w_{K'}(a_i)-\na w_{K''}(a_i)}^2.
\end{align*}
Let $e$ be the common edge between $K',K^{''}\in\om_{a_i}$. Since the tangential derivatives of $w_{K'}$ and $w_{K^{''}}$ match on $e$, then we obtain
\[
\abs{\na w_{K'}(a_i)-\na w_{K''}(a_i)}^2=\labs{\diff{w_{K'}}{n_e}(a)-\diff{w_{K^{''}}}{n_e}(a_i)}^2=\abs{e}^{-1}\nm{\jump{\pa w/\pa n}}{L^2(e)}^2,
\]
where we have used the fact that $\na w$ is a piecewise linear function. Combining
the above two inequalities and using the fact that $h_K\simeq h_{K'}$ because $\mc{T}_h$ is locally quasi-uniform, we obtain
\begin{equation}\label{eq:jump}
\abs{\phi}_{H^m(K)}^2\le Ch_K^{3-2m}\sum_{e\in\mc{E}_{K}}\nm{\jump{\pa\vpi_C v/\pa n}}{L^2(e)}^2,
\end{equation}
where $\mc{E}_{K}$ is the set of the edges emanating from the vertices of $K$. Using the trace inequality~\eqref{eq:polytrace}, we obtain
\begin{equation}\label{eq:spechtintmed}
\abs{\phi}_{H^m(K)}^2\le Ch_K^{2-2m}\sum_{e\in\mc{E}_{K}}\nm{\na\vpi_C v}{L^2(\mc{T}_e)}^2,
\end{equation}
where $\mc{T}_e$ is the triangles contain the edge $e$. The estimate~\eqref{eq:spechtstable} follows by summing up all the elements $K$ and using the following estimate for the interpolant $\vpi_C$.
\begin{equation}\label{eq:scottzhang}
\nm{v-\vpi_C v}{H^m(\Om)}\le Ch^{1-m}\nm{\na v}{L^2(\Om)}\qquad m=0,1.
\end{equation}

Next, for any $v\in H_0^2(\Om)$ and for any $e\in\mc{E}_{K}$ with $K\in\mc{T}_h$, we write
\begin{equation}\label{eq:coreiden}
\jump{\pa\vpi_C v/\pa n}=\jump{\pa(\vpi_C v-v)/\pa n},
\end{equation}
which together with the trace inequality~\eqref{eq:eletrace} yields that for $s=2,3$ and $0\le m\le s$,
\begin{align*}
\abs{\phi}_{H^m(K)}^2&\le Ch_K^{3-2m}\sum_{e\in\mc{E}_{K}}
\Bigl(h_K^{-1}\nm{\na(v-\vpi_C v)}{L^2(\mc{T}_e)}^2+h_K\nm{\na^2(v-\vpi_C v)}{L^2(\mc{T}_e)}^2\Bigr)\\
&\le Ch_K^{2s-2m}\abs{v}_{H^s(\om_K)}^2,
\end{align*}
where $\om_K$ is the set of elements belong to $\mc{T}_h$ that have nonempty intersection with $K$. Summing up all $K\in\mc{T}_h$ and using the interpolation estimate for $\vpi_C$
\[
\nm{v-\vpi_C v}{H^m(\Om)}\le Ch^{s-m}\nm{v}{H^s(\Om)}\qquad\text{for all\quad}0\le m\le s
\]
gives~\eqref{eq:spechtregint1}.

If $v\notin H_0^2(\Om)$, then the identity~\eqref{eq:coreiden} is invalid for the edge $e$ on the boundary. Summing up~\eqref{eq:spechtintmed} for all the elements abutting the boundary of the domain, we obtain that there exists $c_0$ that depends on $\gamma$ such that
\begin{align*}
\sum_{K\in\mc{T}_h,K\cap\pa\Om\not=\emptyset}\abs{\phi}_{H^m(K)}^2&\le Ch^{2-2m}\sum_{K\in\mc{T}_h,K\cap\pa\Om\not=\emptyset}\nm{\na v}{L^2(\om_K)}^2\\
&\le Ch^{2-2m}\int_{x\in\Om,\text{dis}(x,\pa\Om)\le c_0h}\abs{\na v(x)}^2\dx\\
&\le Ch^{3-2m}\nm{\na v}{H^1(\Om)}^2,
\end{align*}
where we have used the estimate~\eqref{eq:skeleton} with $\eps=c_0 h$. This gives~\eqref{eq:spechtregint2} and completes the proof.
\end{proof}

Similar to Corollary~\ref{coro:ntwnmint}, we may obtain
\begin{coro}\label{coro:spechtwint}
There holds
\begin{equation}\label{eq:spechtwint}
\inf_{v\in V_h^0}\wnm{u-v}\le Ch^{1/2}\nm{f}{L^2}.
\end{equation}
\end{coro}

\begin{proof}
Let $v=I_h^0u=(I_h^0u_1,I_h^0u_2)$, using~\eqref{eq:spechtint3} with $s=3/2$, we obtain
\begin{equation}\label{eq:spechtstart}
\nm{\na(u-v)}{L^2}\le\nm{\na(I-I_h^0)u}{L^2}\le Ch^{1/2}\nm{u}{H^{3/2}}\le Ch^{1/2}\nm{f}{L^2},
\end{equation}
where we have used~\eqref{eq:regfrac} in the last step.

Note that $u\in H_0^2(\Om)$, using~\eqref{eq:spechtregint2} with $m=s=2$ and $m=2,s=3$, we obtain
\begin{align*}
\iota\nm{\na^2(u-I_h^0 u)}{L^2}&=\iota\nm{\na^2(u-I_h^0 u)}{L^2}^{1/2}\nm{\na^2(u-I_h^0 u)}{L^2}^{1/2}\\
&\le C\iota h^{1/2}\nm{u}{H^2}^{1/2}\nm{u}{H^3}^{1/2}\\
&\le Ch^{1/2}\nm{f}{L^2},
\end{align*}
where we have used~\eqref{eq:estfrac0} in the last step. A combination of the above two inequalities yields~\eqref{eq:spechtwint}.
\end{proof}
\subsection{Error estimates of the tensor products of the NTW element and the Specht triangle}
Both the NTW element and the Specht triangle are H$^1-$conforming. Therefore, $V_h^0\subset [H_0^1(\Om)]^2$, the jump term in~\eqref{eq:diskorn} vanishes, and we obtain, for any $v\in V_h^0$,
\[
\nm{v}{H_h^2}^2\le C\Lr{\nm{\na_h\eps(v)}{L^2}^2+\nm{\eps(v)}{L^2}^2+\nm{v}{L^2}^2}.
\]
Noting $v\in [H_0^1(\Om)]^2$, and using the {\em Poincar\'e inequality} and the first Korn's inequality~\eqref{eq:1stkorn}, we obtain
\[
\nm{v}{L^2}^2\le C_p^2\nm{\na v}{L^2}^2\le 2C_p^2\nm{\eps(v)}{L^2}^2.
\]
Combining the above two inequalities, we obtain
\begin{equation}\label{eq:hkorn1}
\nm{v}{H_h^2}^2\le C\Lr{\nm{\na_h\eps(v)}{L^2}^2+\nm{\eps(v)}{L^2}^2}\qquad\text{for all\quad}v\in V_h^0.
\end{equation}
This inequality immediately implies the coercivity of $a_h$, and hence the wellposedness of the discrete problem. The following coercivity inequality with respect to the energy norm is crucial for the error estimate.
\begin{lemma}\label{lema:coer1}
There holds
\begin{equation}\label{eq:coer1}
a_h(v,v)\ge (2-\sqrt2)\mu\wnm{v}^2\qquad\text{for all\quad}v\in V_h^0.
\end{equation}
\end{lemma}

\begin{proof}
For any $v\in V_h^0$,
\[
a_h(v,v)\ge 2\mu\Lr{\nm{\eps(v)}{L^2}^2+\iota^2\nm{\na_h\eps(v)}{L^2}^2}.
\]
Using the first Korn's inequality~\eqref{eq:1stkorn} for any $v\in V_h^0$ because $V_h^0\subset [H_0^1(\Om)]^2$,
\[
2\nm{\eps(v)}{L^2}^2\ge\nm{\na v}{L^2}^2,
\]
and using~\eqref{eq:algkorn} for the strain gradient term, we obtain
\[
\nm{\na_h\eps(v)}{L^2}^2\ge(1-1/\sqrt2)\nm{\na_h^2 v}{L^2}^2.
\]
Combining the above three inequalities, we obtain~\eqref{eq:coer1}.
\end{proof}

We are ready to prove the error estimate for the tensor products of the NTW element and the Specht triangle.
\begin{theorem}\label{thm:main1}
There exits $C$ such that
\begin{equation}\label{eq:ferr1}
\wnm{u-u_h}\le Ch^{1/2}\nm{f}{L^2}.
\end{equation}
\end{theorem}

\begin{proof}
By the theorem of {\sc Berger, Scott, and Strang}~\cite{Berge:1972}, we have, there exists $C$ depends on $\lambda$ and
$\mu$ such that
\begin{equation}\label{eq:berge1}
\wnm{u-u_h}\le C\Lr{\inf_{v\in V_h^0}\wnm{u-v}+\sup_{w\in V_h^0}\dfrac{E_h(u,w)}{\wnm{w}}},
\end{equation}
where $E_h(u,w)=a_h(u,w)-\dual{f}{w}$.

By~\cite[Theorem 4]{LiMingShi:2017}, we rewrite the consistency functional $E_h$ as
\[
E_h(u,w)=\sum_{e\in\mc{E}_h}\int_en_in_j\tau_{ijk}\jump{\pa_n w_k}\md\,t,
\]
where $\tau_{ijk}=\iota^2\sigma_{jk,i}$ with stress $\sigma=\mb{C}\eps(u)$. For any $w\in V_h^0$, we have, for any $e\in\mc{E}_h$,
\[
\int_e\jump{\pa_n w_k}\md\,t=0.
\]
By the trace inequality~\eqref{eq:eletrace} and using the regularity estimate~\eqref{eq:estfrac0}, we obtain
\[
\abs{E_h(u,w)}\le Ch^{1/2}\iota^2\nm{\na^2 u}{L^2}^{1/2}\nm{\na^2 u}{H^1}^{1/2}\nm{\na_h^2 w}{L^2}\le Ch^{1/2}\nm{f}{L^2}\wnm{w}.
\]
Substituting the above inequality, the interpolation estimates~\eqref{eq:ntwinter2} and~\eqref{eq:spechtregint2} into the right-hand side of~\eqref{eq:berge1}, we obtain the desired error estimate~\eqref{eq:ferr1}.
\end{proof}
\begin{remark}\label{rk:ferr}
If the solution $u$ has no layer, then we use the interpolation estimates~\eqref{eq:ntwinter1} and~\eqref{eq:spechtinter1}, and proceeding along the same line that leads to~\cite[Theorem 4]{LiMingShi:2017}, and we obtain
\begin{equation}\label{eq:ferr2}
\wnm{u-u_h}\le C(h^2+\iota h)\abs{u}_{H^3}\qquad\text{NTW and the Specht triangle}.
\end{equation}
\end{remark}
\section{The Tensor Product of Morley's Triangle with Variational Prime}
In~\cite{Tai:2001}, Morley's triangle~\cite{Morley:1968} was proved to be divergent for solving a fourth-order singular perturbation problem, which may be regarded as a scalar version of the strain gradient elasticity model~\eqref{eq:sgbvp}. {\sc Wang, Xu and Hu}~\cite{Wang:2006} proved that Morley's triangle uniformly converges for solving this problem with a modified bilinear form. Based on the discrete H$^2-${\em Korn's inequality} proved in the last section, the tensor product of Morley's triangle can be applied to the strain gradient elasticity problem with a modified bilinear form as in~\cite{Wang:2006}. Morley's triangle is defined as
\[
\left\{\begin{aligned}
P_K&=\mb{P}_2(K),\\
\Sigma_K&=\{p(a_i),\negint_{e_i}\pa_n p,\;i=1,2,3.\}.
\end{aligned}\right.
\]
The modified bilinear form $a_h$ is defined as
\[
a_h(v,w){:}=(\C\eps(\pi_1v),\eps(\pi_1w))+\iota^2(\D\na_h\eps_h(v),\na_h\eps_h(w)),
\]
where $\pi_1$ is the standard linear interpolant. The corresponding energy norm is $\wnm{v}{:}=\nm{\na\pi_1v}{L^2}+\iota\nm{\na_h^2v}{L^2}$.

The approximation space is defined as $V_h^0=[X_h]^2$ with
\[
X_h{:}=\set{v\in L^2(\Om)}{v|_K\in P_K,\Sigma_K\text{\;vanishes for all\;}K\cap\pa\Om\not=\emptyset}.
\]

The approximation problem reads as: Find $u_h\in  V_h^0$ such that
\begin{equation}\label{eq:morleyapp}
a_h(u_h,v)=\dual{f}{\pi_1 v}\qquad\text{for all\quad} v\in V_h^0.
\end{equation}
\begin{lemma}\label{lema:interpolate}
There exists $C$ such that
\begin{equation}\label{eq:morleyinter}
\inf_{v\in V_h^0}\wnm{u-v}\le C\Lr{h^{1/2}\wedge\iota^{1/2}}\nm{f}{L^2}.
\end{equation}
\end{lemma}

\begin{proof}
Let $\vpi$ be the interpolation operator for Morley's triangle, and define $v=\vpi u=(\vpi u_1,\vpi u_2)$. Note that
\[
\pi_1\vpi u=\pi_1 u.
\]
We immediately conclude that
\[
\wnm{u-\vpi u}=\iota\nm{\na_h^2(u-\vpi u)}{L^2}=\iota\nm{(I-\vpi_0)\na^2 u}{L^2}\le\iota\nm{\na^2 u}{L^2}.
\]

On the other hand, we may have
\begin{align*}
\wnm{u-\vpi u}&=\iota\nm{(I-\vpi_0)\na^2 u}{L^2}=\iota\nm{(I-\vpi_0)\na^2 u}{L^2}^{1/2}\nm{(I-\vpi_0)\na^2 u}{L^2}^{1/2}\\
&\le\iota(h/\pi)^{1/2}\nm{\na^2 u}{L^2}^{1/2}\nm{\na^3u}{L^2}^{1/2}.
\end{align*}
Combining the above two inequalities and using the regularity estimate~\eqref{eq:estfrac0}, we obtain~\eqref{eq:morleyinter}.
\end{proof}

Similar to~\eqref{eq:coer1}, we may prove the following coercivity inequality in the energy norm.
\begin{lemma}
For any $v\in V_h^0$, there holds
\begin{equation}\label{eq:coer2}
a_h(v,v)\ge\dfrac{\mu}2\wnm{v}^2.
\end{equation}

Moreover, there exists a unique solution $u_h\in V_h^0$ of Problem~\eqref{eq:morleyapp}.
\end{lemma}

\begin{proof}
Using the first Korn's inequality for $\pi_1 v\in [H_0^1(\Om)]^2$, we obtain
\[
\nm{\eps(\pi_1v)}{L^2}^2\ge\dfrac12\nm{\na\pi_1v}{L^2}^2.
\]
Using~\eqref{eq:algkorn} for the strain gradient term, we obtain
\[
\nm{\na_h\eps_h(v)}{L^2}\ge(1-\sqrt2)\nm{\na_h^2 v}{L^2}^2.
\]
The above two inequalities and
\[
a_h(v,v)\ge 2\mu\Lr{\nm{\eps(\pi_1 v)}{L^2}^2+\iota^2\nm{\na_h\eps_h(v)}{L^2}^2}
\]
give~\eqref{eq:coer2}.

If $a_h(v,v)=0$ for any $v\in V_h^0$, then we conclude that
\[
\na_h\pi_1v=0\qquad\text{and\quad}\na_h^2v=0
\]
in a piecewise manner. By $\na_h^2 v=0$, we obtain that $v$ is a piecewise linear vector field. This immediately implies that $\pi_1 v=v$. Using $\na_h\pi_1v=0$, we conclude that $v$ is a constant vector field. Note that $v$ is continuous at
each vertices. Therefore, $v$ is a uniform constant vector filed over the whole domain. Note that $v$ vanishes at the vertices of the boundary, this implies that $v\equiv 0$. This proves the uniqueness of the discrete problem, and the
existence follows from the uniqueness.
\end{proof}

The next theorem gives the rate of convergence of for the approximation problem~\eqref{eq:morleyapp}.
\begin{theorem}
There exists $C$ independent of $\iota$ such that
\begin{equation}\label{eq:morleyaccuracy}
\wnm{u-u_h}\le Ch^{1/2}\nm{f}{L^2}.
\end{equation}
\end{theorem}

\begin{remark}
The half-order rates of convergence~\eqref{eq:ferr1} and~\eqref{eq:morleyaccuracy} for all three elements seem to be the best we can possibly obtain for any finite elements, even if we use a complex conforming finite element, because for $1\le s\le 2$,
\[
\nm{u}{H^s}\le C\iota^{3/2-s}\nm{f}{L^2},
 \]
which may be derived as that leads to~\eqref{eq:regfrac}. Therefore, $\nm{u}{H^s}$ blows up for $s>3/2$. This is also confirmed by the second numerical example in the next part.
\end{remark}

\begin{proof}
We start with~\eqref{eq:berge1}, in which the consistency functional $E_h(u,w){:}=a_h(u,w)-\dual{f}{\pi_1 w}$ changes to
\begin{align*}
E_h(u,w)&=\dual{\mb{C}\eps(\pi_1u-u)}{\eps(\pi_1 w)}+\iota^2\dual{\mb{D}\na\eps(u)}{\na_h\eps_h(w)}\\
&\quad+\sum_{e\in\mc{E}_h}\int_e
\tau_{ijk}\pa_j(\pi_1w_i)n_k\dt.
\end{align*}

By the interpolation estimate~\cite[Theorem 6.1, Example 8.3]{DupontScott:1980}, we obtain
\[
\nm{\na(u-\pi_1u)}{L^2}\le Ch^{1/2}\nm{u}{H^{3/2}}\le Ch^{1/2}\nm{f}{L^2},
\]
where we have used the regularity estimate~\eqref{eq:regfrac}. Using the above estimate, the first term in the right-hand side of $E_h$ may be bounded as
\begin{align*}
\abs{\dual{\mb{C}\eps(\pi_1u-u)}{\eps(\pi_1 w)}}&\le 2(\mu+\lambda)\nm{\na(u-\pi_1 u)}{L^2}\nm{\eps(\pi_1w)}{L^2}\\
&\le Ch^{1/2}\nm{f}{L^2}\wnm{w}.
\end{align*}

Invoking~\eqref{eq:regularity} again, we bound the second term in the right-hand side of $E_h$ as
\[
\iota^2\abs{\dual{\mb{D}\na\eps(u)}{\na_h\eps_h(w)}}\le 2(\mu+\lambda)\iota^2\nm{\na^2 u}{L^2}\nm{\na_h^2 w}{L^2}\le C\iota^{1/2}\nm{f}{L^2}\wnm{w}.
\]

The third term in the right-hand side of $E_h$ can be decomposed into two terms:
\begin{align*}
\sum_{e\in\mc{E}_h}\int_e
\tau_{ijk}\pa_j(\pi_1w_i)n_k\dt&=\sum_{e\in\mc{E}_h}\int_e
\tau_{ijk}\pa_j(\pi_1w_i-w_i)n_k\dt+\sum_{e\in\mc{E}_h}\int_e
\tau_{ijk}\pa_jw_in_k\dt\\
&={:}I+J
\end{align*}

Using the trace inequality~\eqref{eq:eletrace} and the regularity estimate~\eqref{eq:estfrac0}, we obtain
\begin{align*}
\abs{I}&\le C\iota^2\sum_{K\in\mc{T}_h}\nm{\na^2 u}{L^2(\pa K)}\nm{\na(w-\pi_1w)}{L^2(\pa K)}\\
&\le C\iota^2\sum_{K\in\mc{T}_h}\nm{\na^2 u}{L^2(K)}^{1/2}\nm{\na^2 u}{H^1(K)}^{1/2}\Lr{h_K^{-1/2}\nm{\na(w-\pi_1w)}{L^2(K)}+h_K^{1/2}\nm{\na^2 w}{L^2(K)}}\\
&\le Ch^{1/2}\iota^2\nm{\na^2 u}{L^2}^{1/2}\nm{\na^2 u}{H^1}^{1/2}\nm{\na_h^2 w}{L^2}\\
&\le Ch^{1/2}\nm{f}{L^2}\wnm{w}.
\end{align*}
Proceeding along the same line, we bound $J$ as
\[
\abs{J}\le Ch^{1/2}\nm{f}{L^2}\wnm{w}.
\]

Summing up all the above estimates and using the interpolation estimate~\eqref{eq:morleyinter} and noting the fact $\iota\le h$, we obtain~\eqref{eq:morleyaccuracy}.

On the other hand, if $\iota>h$, we may start from another representation of $E_h$:
\begin{align*}
E_h(u,w)&=\dual{\mb{C}\eps(\pi_1u-u)}{\eps(\pi_1 w)}-\sum_{K\in\mc{T}h}\int_K\diff{}{x_i}\tau_{ijk}\diff{}{x_j}(w_k-\pi_1w_k)\dx\\
&\quad+\sum_{e\in\mc{E}_h}\int_e
\tau_{ijk}\diff{w_i}{x_j}n_k\dt.
\end{align*}
The first term and the third term may be bounded as before, while the second term is bounded as
\begin{align*}
\abs{\sum_{K\in\mc{T}h}\int_K\diff{}{x_i}\tau_{ijk}\diff{}{x_j}(w_k-\pi_1w_k)\dx}&\le
Ch\iota^2\nm{u}{H^3}\nm{\na_h^2w}{L^2}\\
&\le Ch\iota^{-1/2}\nm{f}{L^2}\wnm{w}\\
&\le Ch^{1/2}\nm{f}{L^2}\wnm{w},
\end{align*}
where we have used the fact that $\iota>h$ in the last step. This yields~\eqref{eq:morleyaccuracy} and finishes the proof.
\end{proof}
Similar to Remark~\ref{rk:ferr}, we may also obtain the following error bound for the solution without layers
\begin{equation}\label{eq:morleyacc1}
\wnm{u-u_h}\le C(h^2+\iota h)\abs{u}_{H^3}.
\end{equation}
\section{Numerical experiments}
In this part, we report the performance of the proposed elements, i.e., the tensor products of the NTW element and the Specht triangle, and the modified Morley's triangle. We test both the accuracy of the elements and the numerical pollution effect for a solution with strong boundary layer. In all the examples, we let $\Om=(0,1)^2$ and set $\lam=10,\mu=1$, and the initial unstructured mesh is generated by the function "initmesh" of the partial differential equation toolbox of MATLAB. The mesh consists of $220$ triangles and $127$ vertices, and the maximum mesh size is $h=1/8$; See Figure~\ref{mesh}. The mesh is refined by splitting each triangle into four congruent triangles. The maximum mesh size of the finest mesh is $h=1/256$.
\begin{figure}[htbp]
\centering
\includegraphics[width=8cm, height=6cm]{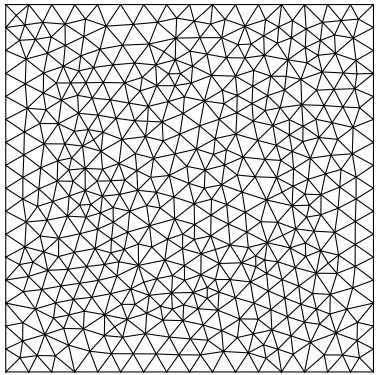}
\caption{Plots of the unstructured mesh with maximum mesh size $h=1/8$.}\label{mesh}
\end{figure}
\subsection{First example}
This example is to test the accuracy of the proposed elements. We let $u=(u_1,u_2)$ with
\[
\left\{
\begin{aligned}
&u_1=\lr{\exp(\cos2\pi x)-\exp(1)}\lr{\exp(\cos2\pi y)-\exp(1)}, \\ &u_2=(\cos2\pi x-1)(\cos4\pi y-1).
\end{aligned}\right.
\]
The source term $f$ is computed by~\eqref{eq:sgbvp}$_1$. The solution has no layer, and we may observe the rates of convergence in the energy norm. In Table~\ref{tab:1}, Table~\ref{tab:2} and Table~\ref{tab:3}, we report the rates of convergence for all the elements in relative energy norm $\wnm{u-u_h}/\wnm{u}$ for different $\iota$. We observe that the rates of convergence appear to be linear when $\iota$ is large, while it turns out to be quadratic when $\iota$ is close to zero, which is consistent with the theoretical predications~\eqref{eq:ferr2} and~\eqref{eq:morleyacc1}. It seems that the modified Morley's triangle is less accurate as the tensor products of the NTW element and the Specht triangle, in particular when $\iota$ is close to zero. This is due to the fact that the modified Morley's triangle degenerates to the standard linear element when $h$ tends to zero.
\begin{table}[htbp]
\caption{Rates of convergence of NTW triangle.}\label{tab:1}
\begin{tabular}{lllllll}
\hline\noalign{\smallskip}
$\iota\backslash h$ &$1/8$ &$1/16$ &$1/32$ &$1/64$ &$1/128$ &$1/256$\\
\hline\noalign{\smallskip}
1e+0 &3.43e-01&1.65e-01&8.06e-02&3.98e-02&1.98e-02&9.88e-03\\
\noalign{\smallskip}
\text{rate}&&1.05&1.04&1.02&1.01&1.00\\
\noalign{\smallskip}
1e-2 &4.80e-02&1.90e-02&8.72e-03&4.24e-03&2.10e-03&1.05e-03\\
\noalign{\smallskip}
\text{rate}&&1.34&1.12&1.04&1.01&1.00\\
\noalign{\smallskip}
1e-4 &3.64e-02&9.56e-03&2.45e-03&6.19e-04&1.56e-04&4.00e-05\\
\noalign{\smallskip}
\text{rate}&&1.93&1.97&1.98&1.99&1.96\\
\noalign{\smallskip}
1e-6 &3.64e-02&9.56e-03&2.44e-03&6.17e-04&1.55e-04&3.88e-05\\
\noalign{\smallskip}
\text{rate}&&1.93&1.97&1.99&1.99&2.00\\
\hline
\end{tabular}
\end{table}
\begin{table}[htbp]
\caption{Rate of convergence of the Specht triangle.}\label{tab:2}
\begin{tabular}{lllllll}
\hline\noalign{\smallskip}
$\iota\backslash h$ &$1/8$ &$1/16$ &$1/32$ &$1/64$ &$1/128$ &$1/256$\\
\hline\noalign{\smallskip}
1e+0 &1.99e-01&9.87e-02&4.80e-02&2.36e-02&1.17e-02&5.85e-03\\
\noalign{\smallskip}
\text{rate}&&1.01&1.04&1.02&1.01&1.00\\
\noalign{\smallskip}
1e-2 &3.16e-02&1.21e-02&5.30e-03&2.53e-03&1.25e-03&6.21e-04\\
\noalign{\smallskip}
\text{rate}&&1.39&1.19&1.07&1.02&1.01\\
\noalign{\smallskip}
1e-4 &2.20e-02&5.57e-03&1.39e-03&3.48e-04&8.75e-05&2.26e-05\\
\noalign{\smallskip}
\text{rate}&&1.98&2.00&2.00&1.99&1.95\\
\noalign{\smallskip}
1e-6 &2.20e-02&5.57e-03&1.39e-03&3.47e-04&8.65e-05&2.16e-05\\
\noalign{\smallskip}
\text{rate}&&1.98&2.00&2.00&2.00&2.00\\
\noalign{\smallskip}
\hline
\end{tabular}
\end{table}
\begin{table}[htbp]
\caption{Rates of convergence of the modified Morley's triangle.}\label{tab:3}
\begin{tabular}{lllllll}
\hline\noalign{\smallskip}
$\iota\backslash h$ &$1/8$ &$1/16$ &$1/32$ &$1/64$ &$1/128$ &$1/256$\\
\hline\noalign{\smallskip}
1e+0 &2.00e+00&1.11e+00&5.91e-01&3.04e-01&1.53e-01&7.69e-02\\
\noalign{\smallskip}
\text{rate}&&0.85&0.91&0.96&0.99&1.00\\
\noalign{\smallskip}
1e-2 &3.97e-01&1.80e-01&7.50e-02&3.38e-02&1.64e-02&8.17e-03\\
\noalign{\smallskip}
\text{rate}&&1.14&1.27&1.15&1.04&1.01\\
\noalign{\smallskip}
1e-4 &3.81e-01&1.56e-01&4.96e-02&1.40e-02&3.71e-03&9.62e-04\\
\noalign{\smallskip}
\text{rate}&&1.29&1.65&1.83&1.91&1.95\\
\noalign{\smallskip}
1e-6 &3.81e-01&1.56e-01&4.96e-02&1.39e-02&3.70-03&9.59e-04\\
\noalign{\smallskip}
\text{rate}&&1.29&1.65&1.83&1.91&1.95\\
\noalign{\smallskip}
\hline
\end{tabular}
\end{table}
\subsection{Second example}
In this example, we test the performance of above elements to resolve a solution with strong boundary layer effect. 
We construct a displacement field with a layer as follows. Let $u=(u_1,u_2)$ with
\begin{align*}
u_1=&\Lr{\exp(\sin\pi x)-1-\pi\iota\dfrac{\cosh\frac{1}{2\iota}-\cosh\frac{2x-1}{2\iota}}{\sinh\frac{1}{2\iota}}}\\
&\times\Lr{\exp(\sin\pi y)-1-\pi\iota\dfrac{\cosh\frac{1}{2\iota}-\cosh\frac{2y-1}{2\iota}}{\sinh\frac{1}{2\iota}}},
\end{align*}
and
\[
u_2=\Lr{\sin\pi x-\pi\iota\frac{\cosh\frac{1}{2\iota}-\cosh\frac{2x-1}{\iota}}{\sinh\frac{1}{2\iota}}}
\Lr{\sin\pi y-\pi\iota\frac{\cosh\frac{1}{2\iota}-\cosh\frac{2y-1}{\iota}}{\sinh\frac{1}{2\iota}}}.
\]
It is clear that the $\na u$ has a layer and
\[
\lim_{\iota\rightarrow0}u=u_0=\lr{\lr{\exp(\sin\pi x)-1}\lr{\exp(\sin\pi y)-1},\sin\pi x\sin\pi y},
\]
with
\[
u_0|_{\pa\Om}=0\quad\text{and}\quad\pa_nu_0|_{\pa\Om}\neq0.
\]
%
%
The source term $f$ is still computed from~\eqref{eq:sgbvp}$_1$. We report the rates of convergence for the elements in the relative energy norm $\wnm{u-u_h}/\wnm{u}$ with a fixed $\iota=10^{-6}$ in Table~\ref{tab:3a}. The half-order rates of convergence are observed for all the elements, which is consistent with the theoretical prediction~\eqref{eq:ferr1} and~\eqref{eq:morleyaccuracy}.
\begin{table}[htbp]
\caption{Rates of convergence for $\iota=10^{-6}$.}\label{tab:3a}
\begin{tabular}{lllllll}
\hline\noalign{\smallskip}
$h$ &$1/8$ &$1/16$ &$1/32$ &$1/64$ &$1/128$ &$1/256$ \\
\hline\noalign{\smallskip}

NTW &1.19e-01&8.40e-02&5.91e-02&4.17e-02&2.95e-02&2.08e-02\\
\noalign{\smallskip}
\text{rate}&&0.51&0.51&0.50&0.50&0.50\\
\noalign{\smallskip}

Specht&1.57e-01&1.10e-01&7.70e-02&5.42e-02&3.82e-02&2.70e-02\\
\noalign{\smallskip}
\text{rate}&&0.51&0.51&0.51&0.50&0.50\\
\noalign{\smallskip}

Morley&2.60e-01&1.71e-01&1.22e-01&8.74e-02&6.25e-02&4.45e-02\\
\noalign{\smallskip}
\text{rate}&&0.60&0.49&0.48&0.48&0.49\\
\noalign{\smallskip}
\hline
\end{tabular}
\end{table}
\section{Conclusion}
We prove a new H$^2-$Korn's inequality and its discrete analog,  which is crucial for the well-posedness of a strain gradient elasticity model. Guided by the discrete H$^2-$Korn's inequality, we construct two family of nonconforming elements that converge uniformly with respect to the microscopic materials parameter with optimal rates of convergence. These elements are simpler than those in~\cite{LiMingShi:2017}. We test the accuracy of the proposed elements for both the smooth solution and the solution with strong boundary layer effect. Numerical results confirm the theoretical predictions. The extension of these elements to three dimensional problems seem very interesting and challenging, and we leave it in a forthcoming work.
\begin{appendix}
\section{Proof of Lemma~\ref{lema:affine}}
In this Appendix we prove Lemma~\ref{lema:affine}.
\begin{proof}
Observing that
\[
(m_i\cdot\na v)=h_i\diff{v}{n_i}+s_i\diff{v}{t_i},
\]
where $h_i$ is the length of the altitude of edge $e_i$ and $t_i$ is the unit tangential vector of $e_i$, and $s_i$ is the
distance between $a_{jk}$ and the foot point of the altitude of edge $e_i$. A direct calculation gives $h_i=2\triangle/\ell_i$ and
\[
s_i=\dfrac{\abs{\ell_j^2-\ell_k^2}}{2\ell_i}.
\]
We rewrite
\[
\sum_{i=1}^3\Lr{\negint_{e_i}(m_i\cdot\na)v}\wt{\psi}_i=\sum_{i=1}^3\Lr{h_i\negint_{e_i}\diff{v}{n_i}+\dfrac{s_i}{\ell_i}
\Lr{v(a_j)-v(a_k)}}\wt{\psi}_i,
\]
where $s_i=(\ell_j^2-\ell_k^2)/(2\ell_i)$. Note that if $\ell_k\ge\ell_j$, then $v(a_j)-v(a_k)$ has to be changed to $v(a_k)-v(a_j)$. Therefore,
\begin{align*}
v&=\sum_{i=1}^3v(a_i)\wt{\phi}_i+\sum_{i=1}^3\Lr{\negint_{e_i}(m_i\cdot\na)v}\wt{\psi}_i
+\sum_{i=1}^3v(a_{jk})\wt{\varphi}_i\\
&=\sum_{i=1}^3v(a_i)\Lr{\wt{\phi}_i+\dfrac{s_j}{\ell_j}\wt{\psi}_j-\dfrac{s_k}{\ell_k}\wt{\psi}_k}
+\sum_{i=1}^3\negint_{e_i}\diff{p}{n_i}h_i\wt{\psi}_i+\sum_{i=1}^3v(a_{jk})\wt{\varphi}_i\\
&=\sum_{i=1}^3v(a_i)\Lr{\wt{\phi}_i+\dfrac{s_j}{\ell_j}\wt{\psi}_j-\dfrac{s_k}{\ell_k}\wt{\psi}_k}
+\sum_{i=1}^3\negint_{e_i}\diff{p}{n_i}\psi_i+\sum_{i=1}^3v(a_{jk})\wt{\varphi}_i.
\end{align*}

It remains to check
\[
\phi_i=\wt{\phi}_i+\dfrac{s_j}{\ell_j}\wt{\psi}_j-\dfrac{s_k}{\ell_k}\wt{\psi}_k.
\]
We only check the case when $i=1$, the others can be obtained by
cyclic permutation of the indices.
\begin{equation}\label{eq:9nodeequiv}
\phi_1=\wt{\phi}_1+\dfrac{s_2}{\ell_2}\wt{\psi}_2-\dfrac{s_3}{\ell_3}\wt{\psi}_3.
\end{equation}
Note that
\begin{align*}
\wt{\phi}_1&=\lam_1(2\lam_1-1)+6b_K(1-2\lam_1)+6b_K\lam_1\\
&=\lam_1(2\lam_1-1)+6b_K(1-2\lam_1)+6b_K(1-\lam_2-\lam_3)\\
&=\lam_1(2\lam_1-1)+6b_K(1-2\lam_1)-\dfrac12\wt{\psi}_2-\dfrac12\wt{\psi}_3.
\end{align*}
A direct calculation gives
\[
\dfrac12-\dfrac{s_2}{\ell_2}=\dfrac{\ell_1^2+\ell_2^2-\ell_3^2}{2\ell_2^2}
=-\dfrac{\na\lam_1\cdot\na\lam_2}{\abs{\na\lam_2}^2},
\]
and
\[
\dfrac12+\dfrac{s_3}{\ell_3}=\dfrac{\ell_1^2+\ell_3^2-\ell_3^2}{2\ell_3^2}
=-\dfrac{\na\lam_1\cdot\na\lam_3}{\abs{\na\lam_3}^2}.
\]
This verifies~\eqref{eq:9nodeequiv} and completes the proof.
\end{proof}
\end{appendix}
\bibliographystyle{amsplain}
\bibliography{sg}
\end{document}